\documentclass[10pt]{article}
\usepackage{amssymb, amsmath, latexsym, a4}
\usepackage[latin1]{inputenc}
\usepackage{graphicx}
\usepackage[all]{xy}
\usepackage[usenames]{color}
\usepackage{pgf,tikz}

\usepackage{mathrsfs}

\usetikzlibrary{arrows}

\newenvironment{proof}[1][Proof]{\noindent\textbf{#1.} }{\ \rule{0.5em}{0.5em}}

\newtheorem{De}{Definition}[section]
\newtheorem{Th}[De]{Theorem}
\newtheorem{Pro}[De]{Proposition}

\newtheorem{Co}[De]{Corollary}
\newtheorem{Rem}[De]{Remark}
\newtheorem{Ex}[De]{Example}

\def\d{\delta}

\newbox\pullbackbox
\setbox\pullbackbox=\hbox{\xy 0;<1mm,0mm>: \POS(4,0)\ar@{-} (0,0)
\ar@{-} (4,4)
\endxy}

\newdir{>>}{{}*!/3.5pt/:(1,-.2)@^{>}*!/3.5pt/:(1,+.2)@_{>}*!/7pt/:(1,-.2)@^{>}*!/7pt/:(1,+.2)@_{>}}
\newdir{ >>}{{}*!/8pt/@{|}*!/3.5pt/:(1,-.2)@^{>}*!/3.5pt/:(1,+.2)@_{>}}
\newdir{ |>}{{}*!/-3.5pt/@{|}*!/-8pt/:(1,-.2)@^{>}*!/-8pt/:(1,+.2)@_{>}}
\newdir{ >}{{}*!/-8pt/@{>}}
\newdir{>}{{}*:(1,-.2)@^{>}*:(1,+.2)@_{>}}
\newdir{<}{{}*:(1,+.2)@^{<}*:(1,-.2)@_{<}}

\begin{document}

\title{Differential graded Hopf algebra structure on free symmetric cosimplicial operads.}
\author{Calvin Tcheka$^{1},$ Batkam Mbatchou V. Jacky III$^{2}$,
 Guy R. Biyogmam$^{3}$ } \maketitle

\bigskip
\centerline{$^{2}$ Department of Mathematics, Faculty of
science-University of Dschang} \centerline{Campus Box (237)67
Dschang, Cameroon} \centerline{ {E-mail address}:
calvin.tcheka@univ-dschang.org, jtchekacalvin@gmail.com}

\bigskip

\centerline{$^{1}$ Department of Mathematics, Faculty of
    science-University of Dschang} \centerline{Campus Box (237)67
    Dschang, Cameroon} \centerline{ {E-mail address}: batkamjacky3@yahoo.com }

\bigskip

\centerline{$^{3}$ Department of Mathematics, Georgia College and
State University, USA.} \centerline{Campus Box 17, Milledgeville}
\centerline{ {E-mail address}: guy.biyogmam@gcsu.edu}
\bigskip
\date{}

\centerline{\bf Abstract} Motivated by the recent work of
Batkam-Tcheka on pointed multiplicative operads, we construct in
this paper new chain complex algebras and two distinct bicomplex
algebra structures on a free symmetric connected multiplicative differential
graded operad. Furthermore, we focus on the non-differential graded
case and construct a differential graded Hopf algebra structure
using the odot product together with an analogue of the
Alexander-Whitney homomorphism and a compatible differential. As a
consequence, we extend the Malvenuto-Reutenauer result by showing
that every free symmetric connected multiplicative operad naturally
carries a differential graded Hopf algebra structure.
\bigskip


\bigskip

\bigskip

{\bf Acknowledgements} \\
Authors are grateful to Fr\'ed\'eric Patras for many useful
discussions. The first author is grateful to the hospitality of
"Max-Planck-Institut f$\ddot{u}$r Mathematik" in
Bonn(MPIM-Bonn-2026), where this paper was completed.
\bigskip
\section{ Introduction}
 Operads are objects habitually used to describe algebraic structures in monoidal categories.
 They are very important in categories with a good notion of homotopy where they
are useful for the study of homotopy invariant algebraic structures
and hierarchies of higher homotopies. Although the premises of
operads can be found in the paper of Lazard (\cite{La}) entitled group
laws and analyzers, operads were introduced in the 1970s at roughly
the same time by May (\cite{May}) and by Boardman and Vogt (\cite{B. V.})
within the context of algebraic topology and the study of loop
spaces. This theory was profoundly revitalized when Ginzburg and
Kapranov (\cite{G-K}) realized that a duality phenomenon observed by
Kontsevich in the study of graph complexes (\cite{Kont}) could be
formalized in terms of an operadic duality relation, namely Koszul duality.
This development systematically paved the way for the application of operads in
algebra.
 Indeed, the classical categories of algebras: commutative
algebras, associative algebras, and Lie algebras are associated with
operads respectively denoted by {\it Com, As} and {\it Lie.}
Furthermore the operads {\it Com} and {\it Lie} are Koszul duals of
one another within the theory of Ginzburg-Kapranov, while the operad
{\it As} is self-dual. This result provided a general framework for
interpreting the correspondences between the
multiplicative structures arising in the homology of commutative algebras
and those appearing in the homology of Lie algebras. Such correspondences had notably
been exploited in the construction of models for the rational homotopy of spaces.\\
In this work, we will mainly be concerned with differential graded
operads (dg-operads for short), $((\mathcal{O}(n))_{r},
 (\mathcal{O}(n))_{r}\stackrel{d^{\mathcal{O}(n)}_{r}}\longrightarrow(\mathcal{O}(n))_{r-1})_{r\geq 0},$
which describes many familiar algebraic structures, such as
commutative, associative, Lie and Poisson algebras. Assuming that
these dg-operads are free (see subsection 2.3), connected and
multiplicative (see Definition 2.6); and considering
$$\begin{array}{ccc}
        (\mathcal{O}(n))_{r}& \stackrel{\partial^{r}_{n}}\longrightarrow & (\mathcal{O}(n-1))_{r}\\
       x & \longmapsto &
         \partial^{r}_{n}=\sum\limits_{i=1}^n(-1)^{i} s_{i}^{n-1},
    \end{array}$$
 the boundary operator and
        $$\begin{array}{ccc}
        (\mathcal{O}(n))_{r}& \stackrel{\delta^{r}_{n}}\longrightarrow & (\mathcal{O}(n+1))_{r}\\
       x & \longmapsto &
        \delta^{r}_n=\sum\limits_{i=0}^{n+1}(-1)^i D_i^{n+1},
    \end{array}$$
    the coboundary operator  as well as the $\odot$ product, all deriving either from the multiplicative property or the connectedness property of
    $((\mathcal{O}(n))_{r}, d^{\mathcal{O}(n)}_{r}),$
    we establish that the triples $(((\mathcal{O}(n))_{r},
d^{\mathcal{O}(n)}_{r}),
\partial, \odot)$ and $(((\mathcal{O}(n))_{r},
d^{\mathcal{O}(n)}_{r}), \delta, \odot)$ are bicomplex algebras
while the pair $(((\mathcal{O}(n))_{r}, d^{\mathcal{O}(n)}_{r}),
\odot)$ is a chain complex algebra.\\
Besides the $\odot$ product, we define a new product
$$\begin{array}{ccc}
    (\mathcal{O}(p) \otimes \mathcal{O}(q))_{r} = \bigoplus\limits_{u + v = r}(\mathcal{O}(p))_{u}\otimes (\mathcal{O}(q))_{v} &
    \stackrel{\nabla}\longrightarrow &(\mathcal{O}(p+q))_{r} \\
    x\otimes y & \longmapsto & x\nabla y= \sum\limits_{\sigma\in \Sigma_{(p,q)}}sgn(\sigma)(x\odot
    y)\cdot\sigma^{-1},
 \end{array}$$
called the twisted odot product. This product endows  $((\mathcal{O}(n))_{r},
d^{\mathcal{O}(n)}_{r})$ with a differential graded algebra
structure, which we call a twisted  differential graded algebra. We emphasize
here that this twisted  differential graded algebra extends the
Malvenuto-Reutenauer twisted algebra structure on the associative
operad, $\mathcal{A}_{ss}.$\\
We end our construction by defining a differential graded Hopf
algebra structure on a free symmetric cosimplicial operad. This construction leads naturally to several algebraic structures on
$\mbox{End}(\mathcal{O}(n)),$ including convolution algebra and
Lie algebra structures.\\
This work is organized as follows: In section 2, we provide general
recollections on free symmetric cosimplicial dg operads. In section
3, we recall some result developed in \cite{B-T} and define our new
algebraic structures. In section 4, we focus on the non-differential
graded case and construct a differential graded Hopf algebra
structure using the odot product together with an analogue of the
Alexander-Whitney homomorphism and a compatible differential.
Finally, Section 5 is devoted to applications.

 \section{Preliminaries}
 The goal of this section is to recall some useful definitions, preliminary results and properties concerning operads
 over an additive symmetric monoidal category $\mathcal{C}$ with countable sums and quotients by finite
 group actions on its objects. As a guiding example, the reader may keep in mind the category of vector spaces, $Vect_\mathbb{K},$
 over an arbitrary ground field $\mathbb{K}.$ Let $\sum=\{\sum_n, n\in\mathbb{N}\}$ denote the set of
 symmetric groups. An operad is said to be symmetric ($\sum$-operad for short), or simply an operad, if it is endowed with a right
 action of $\sum.$ Otherwise, it is called non-symmetric (non-$\sum$-operad, abbreviated ns-operad).\\For more details, reader may refer to [\cite{M-V-B},\cite{May},\cite{L-V}]
 \subsection{Substitution in permutations}
 Let n,m be two nonzero integers, $\alpha=(\alpha_1,\cdots,\alpha_n)\in\sum_n$ and $\beta=(\beta_1,\cdots,\beta_m)\in\sum_m$. For $i\in [n]=\{1,2,\cdots,n\}$, the substitution of $\beta$ at the i-th position of $\alpha$ is the permutation denoted $sub_i(\alpha,\beta)\in\sum_{n+m-1}$ and defined by $sub_i(\alpha,\beta):=\theta,$ such that
 $\theta=(\theta_1,\cdots,\theta_{i-1},\alpha_i+\beta_1-1,\cdots,\alpha_i+\beta_m-1,\theta_{i+1},\cdots,\theta_n)$ with
 $$ \theta_j=\left\{\begin{array}{lll}
     \alpha_j\quad \mbox{if} \quad \alpha_j<\alpha_i \\
    \alpha_j+m-1 \quad \mbox{otherwise}
 \end{array}
 \right.$$
\begin{Ex}
    Consider $\alpha=(4,8,2,5,7,1,3,6)\in\sum_8$ and $\beta=(3,5,1,2,4)\in\sum_5$. The substitution of $\beta$ at the $4$-th position of $\alpha$
    is:\\
    $ sub_4(\alpha,\beta)=(4,12,2,7,9,5,6,8,11,1,3,10)\in \sum_{12} $
\end{Ex}
 \subsection{Operad}
 Besides the classical definition of an operad based on a single total composition map $\gamma$, there is an alternative and equivalent axiomatic approach, extensively developed by Markl (see e.g  \cite{Ma}). This approach is based on more elementary operations, known as the partial composition, and is often more suitable for combinatorial and homological algebra. Moreover, the partial composition, whose definition will be recalled below, takes advantage of the fact that it suffices to understand the composition of two operations in order to describe the entire structure.
  \begin{De}(Partial composition of  $\sum$-differential graded operads)\\
     A $\sum$-dg operad ($\Sigma$-dg-operad for short) is a collection of right
        $\sum$-differential positive graded $\mathbb{K}$-vector
        spaces,
        $$\{(\mathcal{O}(n))_{k\geq 0}, (\mathcal{O}(n))_{k}\stackrel{d_{k}^{\mathcal{O}}}\longrightarrow (\mathcal{O}(n))_{k-1}: n\geq 0 \} $$ equipped with a unit
        element $1_{\mathcal{O}}\in(\mathcal{O}(1))_{0}$, and a family of partial composition maps: for $m, n, l, r\in \mathbb{N}, i\in [n]$,
        $$\begin{array}{ccc}
            (\mathcal{O}(n))_{r} \otimes (\mathcal{O}(m))_{l} & \stackrel{\circ_i}\longrightarrow & (\mathcal{O}(n+m-1))_{r+l}\\
            x \otimes y& \longmapsto & x\circ_i y.
        \end{array}$$
        These maps are subject to the following axioms:
   %==============================================================================================================================================================
        \begin{enumerate}
            \item[(i)] Associative property: for $x\in(\mathcal{O}(n))_{u}$, $y\in(\mathcal{O}(m))_{v}$, $z\in(\mathcal{O}(l))_{r}$, the following relations hold
            $$\left\{\begin{array}{lll}
                (x\circ_i y)\circ_{i+j-1}z=x\circ_i (y\circ_jz), \quad \mbox{for} \quad i\in [n],~ j\in [m]\\
                \\
                (x\circ_i y)\circ_{k+m-1}z=(x\circ_k z)\circ_i y, \quad \mbox{for} \quad  1\le i<k\le n;
            \end{array}
            \right.$$
            \item[(ii)] Unit property: the unit $1_{\mathcal{O}}\in(\mathcal{O}(1))_{0}$ verifies the following identities: for $x\in(\mathcal{O}(n))_{r}$,
            $n\ge 1$,
            $$x\circ_i1_{\mathcal{O}}= x =1_{\mathcal{O}}\circ_1 x,~~~for ~~~i\in [n] ~~~and ~~~ d^{\mathcal{O}}_{0}(1_{\mathcal{O}})= 0$$
            \item[(iii)] Equivariance property: for $x\in(\mathcal{O}(n))_{l}$, $y\in(\mathcal{O}(m))_{r}$,
             $\sigma\in \sum_n$, $\tau\in \sum_m$, the following identity holds
            $$(x.\sigma)\circ_i(y.\tau)=(x\circ_{\sigma(i)}y)\cdot sub_i(\sigma,\tau),~~~for ~~~i\in [n].$$
            \item[(iv)] For $x\in(\mathcal{O}(n))_{u}$,
            $y\in(\mathcal{O}(m))_{v}$, $\sigma\in \sum_n,$ the following
            relation holds: $$d_{u+v}^{\mathcal{O}}(x\circ_i
            y)= d_{u}^{\mathcal{O}}(x)\circ_i y + (-1)^{|x|} x\circ_i
            d_{v}^{\mathcal{O}}(y), ~d_{u}^{\mathcal{O}}(x\cdot \sigma )= d_{u}^{\mathcal{O}}(x)\cdot
\sigma.$$
    \end{enumerate}
 \end{De}
 Observes that in the above definition, each $\circ_i$ describes the fundamental insertion action of the operad element $y\in\mathcal{O}(m)$ into the $i$-th
 position($i\in [n]$) of the operad element $x\in\mathcal{O}(n)$. This approach decomposes the complex total composition of operads, recalled below, into their atomic building blocks.
 \begin{De}(Total Composition of $\sum$-dg operad)\\
     A $\Sigma$-dg-operad is a collection of right
        $\sum$-differential graded $\mathbb{K}$-vector spaces   $$\{(\mathcal{O}(n))_{k\geq 0}, (\mathcal{O}(n))_{k}\stackrel{d_{k}^{\mathcal{O}}}
        \longrightarrow (\mathcal{O}(n))_{k-1}: n\geq 0 \} $$ together with a composition product:
        $$\begin{array}{ccc}
            (\mathcal{O}(k))_{u} \otimes (\mathcal{O}(n_{1}))_{v_{1}} \otimes \cdots \otimes (\mathcal{O}(n_{k}))_{v_{1}}&
            \stackrel{\gamma^{\mathcal{O}}}\longrightarrow & (\mathcal{O}(\sum\limits_{j=1}^{k} n_{j}))_{u+ \zeta}\\
            x \otimes x_1 \otimes \cdots \otimes x_k & \longmapsto &
            \gamma^{\mathcal{O}}(x; x_1,\cdots , x_k) \quad \mbox{with} \quad \zeta= \sum\limits_{j=1}^{k}v_{j}
        \end{array}$$
        which satisfies the associativity, the unity, the equivariance properties (see \cite{M-V-B} and \cite{B-T} for explicit details on these properties)
        and
 $$ \begin{array}{ccc} d_{u+\zeta}^{\mathcal{O}}(\gamma^{\mathcal{O}}(x; x_1,\cdots , x_k))&= &\gamma(d_{u}^{\mathcal{O}}(x); x_1,\cdots , x_k) \\
        & +& \sum\limits_{j=1}^{s-1}(-1)^{u+\zeta^{\prime}}\gamma^{\mathcal{O}}(x; x_1,\cdots
        x_{s-1},d_{v_{s}}^{\mathcal{O}}(x_{s}), x_{s+1}\cdots,
        x_k) \\
        \quad \mbox{with} \quad \zeta^{\prime}= \sum\limits_{j=1}^{s-1}v_{j}& &
         \end{array}$$
 \end{De}
\begin{Rem}
    The partial and the total compositions of $\sum$-dg operads are related as follows:
    \begin{enumerate}
        \item[(a)] From partial composition to total composition:  for  $x\in(\mathcal{O}(n))_{u}$, $y\in(\mathcal{O}(m))_{v}$, and $i\in [n]$, one has:
        $$x\circ_iy=\gamma^\mathcal{O}(x;\overbrace{id,\cdots,id,\underbrace{y}_{i\mbox{-}th position},\cdots,id}^{n\mbox{-}tuple})$$
        \item[(b)] From total composition to partial composition: for  $x\in(\mathcal{O}(n))_{u}$, $y_j\in(\mathcal{O}(n_j))_{v}$, and $j\in [n]$, one has:
        $$\gamma^\mathcal{O}(x;y_1,\cdots,y_n)=(\cdots((x\circ_ny_n)\circ_{n-1}y_{n-1})\cdots)\circ_1y_1$$
    \end{enumerate}
\end{Rem}
 \begin{De}(Homomorphism of $\sum$-dg operads)\\
     Let  $((\mathcal{O}(n))_{k\geq 0}, d^{\mathcal{O}}) $ and $((\mathcal{O^{\prime}}(n))_{k\geq 0}, d^{\mathcal{O^{\prime}}})$
     be two $\Sigma$-dg operads
        with respective associated units $1_\mathcal{O}$ and
        $1_{\mathcal{O^{\prime}}}$. A homomorphism of $\Sigma$-dg operads
        $((\mathcal{O}(n))_{k\geq 0}, d^{\mathcal{O}})\stackrel{f}\longrightarrow((\mathcal{O^{\prime}}(n))_{k\geq 0}, d^{\mathcal{O^{\prime}}})$
        is a collection $$\{((\mathcal{O}(n))_{u},d^{\mathcal{O}})
\stackrel{f_n^{u}}\longrightarrow((\mathcal{O^{\prime}}(n))_{u},
d^{\mathcal{O^{\prime}}})\}_{n\geq 0}^{u\geq 0}$$ of
        $\sum$-differential graded $\mathbb{K}$-vector space homomorphisms such that:
        \begin{enumerate}
            \item[(a)] $f_{1}^{0}(1_{\mathcal{O}})= 1_{\mathcal{O}^{\prime}}$;
            \item[(b)] for $x\in(\mathcal{O}(n))_{u}$, $y\in(\mathcal{O}^{\prime}(m))_{v},$
            $f_{n+m-1}^{u+v}(x\circ_i^{\mathcal{O}}y)= f_n^{u}(x)\circ_i^{\mathcal{O}^{\prime}} f_{m}^{v}(y)$
            with $\circ_i^{\mathcal{O}}$ and $\circ_i^{\mathcal{O}^{\prime}}$, the respective partial compositions in $\mathcal{O}$ and $\mathcal{O}^{\prime}$;
            \item[(c)] $f_n^{u}(x\cdot\sigma)= f_n^{u}(x)\cdot\sigma,$ for some $x\in(\mathcal{O}(n))_{u}$, $\sigma\in S_n$.
            \item[(d)] $f_n^{u-1}\circ d_{u}^{\mathcal{O}(n)}=
            d_{u}^{\mathcal{O}^{\prime}}\circ f^{u}_{n},$
 for $n\geq 0, u\geq 0.$
 \end{enumerate}
 \end{De}

 \begin{De}
    \begin{enumerate}
        \item A $\Sigma$-dg operad, $((\mathcal{O}(n))_{k\geq 0}, d^{\mathcal{O}}),$ is said to be
        \textbf{multiplicative} if there exists a zero-degree two-inputs element $\mu \in(\mathcal{O}(2))_{0}$
        such that $\mu \circ^{\mathcal{O}}_1 \mu= \mu\circ^{\mathcal{O}}_2 \mu$ (\cite{G-V}) and $d_{0}(\mu)= 0$.
        \item Let $((\mathcal{O}(n))_{k\geq 0}, d^{\mathcal{O}})$ be a $\Sigma$-dg operad and
        denote by $(\mathcal{O}(0))_{k}$, for some
        $k,$ the parametrization of the $k$-degree, 0-arity operations.  $((\mathcal{O}(n))_{k\geq 0}, d^{\mathcal{O}})$ is said to be \textbf{pointed}  if there is
            an injective zero-degree differential graded homomorphism $\mathbb{K}\stackrel{\eta}\rightarrow\mathcal{O}(0)$
            which maps the unit $1_{\mathbb{K}}$ to the element denoted $1_0:= \eta(1_{\mathbb{K}})$ and is
            called the base point of the $\Sigma$-dg operad,
             $((\mathcal{O}(n))_{k\geq 0}, d^{\mathcal{O}}).$ In addition $1_0$ satisfies $d^{\mathcal{O}}(1_0)=
             0.$\\ If the homomorphism $\eta$ instead of being injective, is
             an isomorphism, then $((\mathcal{O}(n))_{k\geq 0}, d^{\mathcal{O}})$ is said to be
             \textbf{connected} and $1_0:= \eta(1_{\mathbb{K}}),$ also called the base point of the $\Sigma$-dg operad, satisfies
             $d^{\mathcal{O}}(1_0)= 0.$

    \end{enumerate}
\end{De}
     \begin{Rem}
      \begin{enumerate}
        \item[(i)] The collection of $\Sigma$-connected differential graded operads forms a subcategory of the category
        of operads whose homomorphisms are homomorphisms of
        $\Sigma$-differential graded operads
        $((\mathcal{O}(n))_{k\geq 0}, d^{\mathcal{O}})\stackrel{\psi}\longrightarrow((\mathcal{O^{\prime}}(n))_{k\geq 0}, d^{\mathcal{O^{\prime}}})$
        satisfying: $\psi(1_{0}^{\mathcal{O}})= 1_{0}^{\mathcal{O^{\prime}}},$ with
        $1_{0}^{\mathcal{O}} \in \mathcal{O}(0)$ and $1_{0}^{\mathcal{O^{\prime}}}\in
        \mathcal{P}(0).$
\item[(ii)] As in the case of connected operads (see \cite{L-P}), for every subset
$S\subset [n]= \{1, 2, \cdots, n \},$ with
        cardinality $l<n$, the $\Sigma$-differential graded connected operad $((\mathcal{O}(n))_{k\geq 0}, d^{\mathcal{O}})$
         is equipped with a degeneracy map
        $\mid_S$ defined as follows:
        \begin{center}
            $\begin{array}{ccc}
                \mid_S:(\mathcal{O}(n))_{u} & \longrightarrow & (\mathcal{O}(l))_{u} \\
                p& \mapsto & p\mid_S=p(x_1,...,x_n),
            \end{array}$
            where $ x_i= \left\{\begin{array}{lll}
                1_{\mathcal{O}} \quad \mbox{if} \quad i\in S, \\
                1_{0}  \quad \mbox{if not.}
            \end{array}
            \right.$
        \end{center}

    \end{enumerate}
 \end{Rem}
 \indent The historical example of operads is the endomorphism operad denoted
 here by $ \mathcal{L}_{A}:= End_{A},$ for an object $A$ in the
 category of $\mathbb{K}$-vector spaces.
 \begin{Ex}   Consider a unitary associative chain $\mathbb{K}$-algebra $(\mathcal{A}_{k}, d^{\mathcal{A}}_{k})_{k\geq 0}$ with multiplication
$\mu_{\mathcal{A}}$ and unit $\eta_{\mathcal{A}},$  and let
$((\mathcal{L}^{\prime}_{\mathcal{A}}(n))_{k},
d^{\mathcal{L}^{\prime}_{\mathcal{A}}}_{k})_{k\geq 0}$ be the
$\Sigma$-dg operad of multilinear graded homomorphisms of
$\mathbb{K}$-vector spaces defined by:\\ for all $n\geq 0,$
 $\mathcal{L}^{\prime}_{\mathcal{A}}(n):= \bigoplus\limits_{r\ge 0}(\mathcal{L}^{\prime}_{\mathcal{A}}(n))_{r}$
 with
 $$(\mathcal{L}^{\prime}_{\mathcal{A}}(n))_{r}=\left\{\begin{array}{ll}
    [Hom_{\mathbb{K}}(A^{\otimes n}, A)]_{r} \quad \mbox{for all} \quad n\geq 1,\\ %r\geq 0, \\
   \mathbb{K}\qquad \mbox{if} \quad n= 0.\\
\end{array}
\right.$$ Denote by  $[Hom_{\mathbb{K}}(A^{\otimes n}, A)]_{r},$
the set of $r$-degree multilinear graded
homomorphisms. The differential
 $(\mathcal{L}^{\prime}_{\mathcal{A}}(n))_{r}\stackrel{d^{\mathcal{L}^{\prime}_{\mathcal{A}}}_{r}}\longrightarrow
 \mathcal{L}_{\mathcal{A}}(n))_{r-1}$ is given by $$d^{\mathcal{L}^{\prime}_{\mathcal{A}}}_{r}(\phi)= d^{\mathcal{A}}\circ \phi -
 (-1)^{deg(\phi)} \phi \circ \tilde{d}^{\mathcal{A}},$$
where $deg(\phi)= r$ denotes the degree of $\phi$   and
$$\tilde{d}^{\mathcal{A}}=
\sum\limits_{i=1}^{n}(\overbrace{id_{\mathcal{A}}\otimes
 id_{\mathcal{A}}\otimes\cdots\otimes id_{\mathcal{A}}\otimes\underbrace{d^{\mathcal{A}}}_{i\mbox{-}th
position}\otimes \cdots\otimes id_{\mathcal{A}}}^{n\mbox{-}tuple}),$$ the
differential in $\mathcal{A}^{\otimes n}.$\\
 $((\mathcal{L}^{\prime}_{\mathcal{A}}(n))_{k},
d^{\mathcal{L}^{\prime}_{\mathcal{A}}}_{k})_{k\geq 0}$ is a unitary
connected multiplicative $\Sigma$-dg operad with
\begin{enumerate}
 \item[$\bullet$] the associated
operadic partial composition
$\circ_{i}^{\mathcal{L}^{\prime}_{\mathcal{A}}},$ the substitution
of the value of an operation at the i-th position of an $n$-ary
operation as input,
\item[$\bullet$] its associated multiplication, $\mu_{\mathcal{L}^{\prime}_{\mathcal{A}}}:= \mu_{\mathcal{A}}\in (\mathcal{L}^{\prime}_{\mathcal{A}}(2))(0)$
 satisfies $d^{\mathcal{L}^{\prime}_{\mathcal{A}}}_{0}(\mu_{\mathcal{L}^{\prime}_{\mathcal{A}}})=0,$
\item[$\bullet$] the connectedness of the operad $\mathcal{L}^{\prime}_{A}$ is the consequence of its definition,
\item[$\bullet$] its associated  unit $1_{\mathcal{L}^{\prime}_{\mathcal{A}}}:= id_{\mathcal{A}}: \mathcal{A}\longrightarrow
 \mathcal{A}\in (\mathcal{L}^{\prime}_{\mathcal{A}}(1))(0) $ also satisfies $d^{\mathcal{L}^{\prime}_{\mathcal{A}}}_{0}(1_{\mathcal{L}^{\prime}_{\mathcal{A}}})=0.$
\end{enumerate}
 \end{Ex}
\subsection{$\Sigma$-differential graded Operad representation}
In this subsection, we present a well-known constructive and
algebraic approach to define operads, namely the definition of operads
through generators and relations. This approach consists of describing an operad by specifying a minimal set of generating
operations and relations (see \cite{F},\cite{L},\cite{M} or
\cite{L-V} for more details on the non-graded case).
\subsubsection{Free $\Sigma$-differential graded operads.}
\begin{De}
    Let $\mathcal{Q}=\{\mathcal{Q}(n)= \bigoplus\limits_{r\geq 0}^{}(\mathcal{Q}(n))_{r} ~:~n\ge 0\}$ be a collection of generators.
    A free $\Sigma$-dg operad on $\mathcal{Q}$ is a pair $((\mathcal{F}(\mathcal{Q}), d^{\mathcal{F}(\mathcal{Q})}),\tau)$,
    where $\mathcal{F}(\mathcal{Q})=\{\mathcal{F}(\mathcal{Q}(n))= \bigoplus\limits_{r\geq 0}^{}
    \mathcal{F}((\mathcal{Q}(n))_{r})~:~n\ge 0\}$ is a $\Sigma$-graded operad with the boundary operator $d^{\mathcal{F}(\mathcal{Q})}$
     defined on the generic elements such that for all $n\geq 0,$ for all $r\geq 0$ $d_{r}^{\mathcal{F}(\mathcal{Q})}((\mathcal{Q}(n))_{r})\subseteq
     (\mathcal{Q}(n))_{r-1}$ and $$\tau:=\{(\mathcal{Q}(n))_{r}\stackrel{\tau^{r}_{n}}\longrightarrow\mathcal{F}((\mathcal{Q}(n))_{r})
     ~:~ \tau_{n}^{r-1}\circ d_{r}^{\mathcal{F}(\mathcal{Q}(n))} =
d_{r}^{\mathcal{F}(\mathcal{Q}(n))}\circ \tau^{r}_{n},~n\ge 0,~
r\geq 0\} $$
     is a family of maps that satisfies the universal property: for any other $\Sigma$-dg operad
     $(\mathcal{O}=\{(\mathcal{O}(n))_{r\geq 0}, d^{\mathcal{O}})~:~n\ge 0\} $
     and any family of maps $$f=\{((\mathcal{Q}(n))_{r},d^{\mathcal{F}(\mathcal{Q}(n))})\stackrel{f^{r}_n}\longrightarrow
     ((\mathcal{O}(n))_{r},  d^{\mathcal{O}}): f_{n}^{r-1} \circ d_{r}^{\mathcal{F}(\mathcal{Q}(n))} =
d_{r}^{\mathcal{O}} \circ f^{r}_{n}, n\geq 0,~ r\geq 0 \},$$
     there exists a unique homomorphism of $\Sigma$-dg operads,
     $$\tilde{f}= \{(\mathcal{F}((\mathcal{Q}(n))_{r}), d^{\mathcal{F}(\mathcal{Q}(n))})\stackrel{\tilde{f}_{n}^{r}}\longrightarrow((\mathcal{O}(n))_{r},  d^{\mathcal{O}}), n\geq 0, r\geq 0 \} $$
     such that for each $n\geq 0$ and for each $r\geq 0$, the following diagram commutes
    \begin{center}
        $\xymatrix{
            (\mathcal{Q}(n))_{r} \ar[rd]^{f^{r}_n} \ar[r]^{\tau^{r}_n}& \mathcal{F}((\mathcal{Q}(n))_{r}) \ar[d]^{\exists !\tilde{f}^{r}_n}&\\
            & (\mathcal{O}(n))_{r}& \\
        }$
    \end{center}
\end{De}
\begin{Ex}
    Consider the associative operad $\mathcal{A}_{ss}=\mathbb{K}[\sum]=\bigoplus\limits_{n\ge 0}\mathbb{K}[\sum_n]$.
    Observe that any permutation $\theta \in \sum$ can be written as a successive composition of $\sigma=(12)\in S_2$ by
    itself followed by the action of $\sum$. Moreover any element in $\mathcal{A}_{ss}$ is a $\mathbb{K}$-linear combination
    of elements in $\sum$ and since $\mu_{\mathcal{A}_{ss}}= \sigma=(12)\in(\mathcal{A}_{ss}(2))_{0},$ then $\mathcal{A}_{ss}$ is a free
    $\Sigma$-dg operad concentrated in degree zero, generated by
    $\mu_{\mathcal{A}_{ss}}$ with the zero map as differential.
\end{Ex}
\subsubsection{Representation of $\Sigma$-differential graded operads.}

\begin{De}($\Sigma$-differential graded Operadic ideal)\\
 Let $\mathcal{O}=\{((\mathcal{O}(n))_{r}, d^{\mathcal{O}}): n\ge 0,  r\ge 0\} $, be a $\Sigma$-dg operad and
 $ \mathcal{I}=\{(\mathcal{I}(n))_{r}\subset(\mathcal{O}(n))_{r}: n\geq 0, r\geq 0\}.$
 $\mathcal{I}$ is said to be a $\Sigma$-differential graded ideal ($\Sigma$-dg ideal for short) of $\mathcal{O}$ if the following conditions hold:
 \begin{enumerate}
    \item[(a)] for all $n\ge 0$, $((\mathcal{I}(n))_{r})_{r\geq 0}$ is a $\Sigma$-invariant graded subspace of $\mathcal{O}(n)$
    such that for all $r\geq 0$
    $d_{r}^{\mathcal{O}}((\mathcal{I}(n))_{r})\subset(\mathcal{I}(n))_{r-1};$
    \item[(b)] for all $n\ge 0$, for all $r\geq 0,$ for $x\in(\mathcal{O}(n))_{r}$, $(x_i\in(\mathcal{O}(n_i))_{r_{i}})_{1\le i\le n}$,
    $$\gamma^{\mathcal{O}}(x; x_1,\cdots,x_k)\in(\mathcal{I}(\kappa))_{\varsigma}, \kappa =\sum\limits_{i=1}^{n}n_{i}, \varsigma
    =r + \sum\limits_{i=1}^{n}r_{i},$$
    if $x\in (\mathcal{I}(n))_{r},$ or there exists at least one $i_{0}\in \{1,\cdots, n\} $ such that
    $x_{i_{0}}\in(\mathcal{I}(n_{i_{0}}))_{r_{i_{0}}}.$
 \end{enumerate}
\end{De}
\begin{Rem}
    Let $\mathcal{O}=\{((\mathcal{O}(n))_{r}, d^{\mathcal{O}}): n\ge 0,  r\ge 0\}$ be a $\Sigma$-dg operad and
 $ \mathcal{I}=\{(\mathcal{I}(n))_{r}\subset(\mathcal{O}(n))_{r}: n\geq 0, r\geq
 0\},$ its $\Sigma$-dg ideal. we have the following assertions:
    \begin{enumerate}
        \item[(i)] The sequence
        $(\frac{\mathcal{O}}{\mathcal{I}}, d^{\mathcal{O},\mathcal{I}}) :=\{(\frac{(\mathcal{O}(n))_{r}}{(\mathcal{I}(n))_{r}},
        \frac{(\mathcal{O}(n))_{r}}{(\mathcal{I}(n))_{r}}\stackrel{d^{\mathcal{O},\mathcal{I}}} \longrightarrow \frac{(\mathcal{O}(n))_{r-1}}
        {(\mathcal{I}(n))_{r-1}})
        ~:~n\ge 0, r\ge 0\}$ is
         a $\Sigma$-dg operad called quotient $\Sigma$-dg operad (quotient $\Sigma$-dg operad, for short).
        \item[(ii)] For any other $\Sigma$-dg operad $(\mathcal{O^{\prime}}, d^{\mathcal{O^{\prime}}})$ and
         any homomorphism of $\Sigma$-dg operads,
         $(\mathcal{O}, d^{\mathcal{O}})\stackrel{f}\rightarrow(\mathcal{O^{\prime}}, d^{\mathcal{O^{\prime}}})$,
         there exists a unique homomorphism of $\Sigma$-dg operads,
         $(\frac{\mathcal{O}}{\mathcal{I}}, d^{\mathcal{O},\mathcal{I}})\stackrel{\tilde{f}}\rightarrow(\mathcal{O^{\prime}}, d^{\mathcal{O^{\prime}}})$ such that for
        each $n\ge 0,$ and each $r\ge 0,$ the following diagram commutes:
        \begin{center}
            $\xymatrix{
                ((\mathcal{O}(n))_{r}, d_{r}^{\mathcal{O}}) \ar[rd]^{f^{r}_n} \ar[r]^{\pi^{r}_n}& (\frac{(\mathcal{O}(n))_{r}}{(\mathcal{I}(n))_{r}}, d_{r}^{\mathcal{O},\mathcal{I}}) \ar[d]^{\exists !\tilde{f}^{r}_n}& here ~\pi^{r}_{n}~ is~ the~ canonical~ surjection \\
                &((\mathcal{O}^{\prime}(n))_{r}, d_{r}^{\mathcal{O^{\prime}}})&\\
            }$
        \end{center}
    \end{enumerate}
\end{Rem}
\begin{De}
    Let $(\mathcal{F}(\mathcal{Q}), d^{\mathcal{F}(\mathcal{Q})})$, a free $\Sigma$-dg operad over the collection of
    generators $\mathcal{Q}=\{\mathcal{Q}(n): n\ge 0\} $. Denote by $\langle \mathcal{R}\rangle$
    its $\Sigma$-dg ideal generated by $\mathcal{R}$, a set
    of relations. The quotient $\Sigma$-dg operad
     $(\langle\mathcal{F}(\mathcal{Q}),\mathcal{R}\rangle:=\frac{\mathcal{F}(\mathcal{Q})}{\langle
\mathcal{R}\rangle},
     d^{\mathcal{F}(\mathcal{Q}), \langle \mathcal{R}\rangle}) $ is
    called $\Sigma$-dg operad generated by $\mathcal{Q}$ subject
    to the relations in $\mathcal{R}$.
\end{De}
\begin{Ex}
    Let $\mathcal{Q}=\{(\mathcal{Q}(n))_{r} : n\ge 0, r\ge 0\} $ be a collection of generators subject to the following conditions:\\
    $(\mathcal{Q}(1))_{r}$ and $(\mathcal{Q}(2))_{r}$ are graded sets concentrated in zero degree
with $(\mathcal{Q}(1))_{0}=\{1_{\mathcal{Q}}\},$
$(\mathcal{Q}(2))_{0}= \{\nu\} $, for some symbol $\nu$ such that
$\nu\circ_{1}1_{\mathcal{Q}}= \nu\circ_{2}1_{\mathcal{Q}} = \nu =
1_{\mathcal{Q}}\circ_{1}\nu.$
 $(\mathcal{Q}(n))_{r}=\emptyset$ for $n\neq 1; 2,  r\ge
0$.
 Setting $\mathcal{R}=\{\nu +\nu\cdot (12)=0; (\nu,(1_{\mathcal{Q}},\nu))+(\nu,(1_{\mathcal{Q}},\nu))\cdot (123)+(\nu,(1_{\mathcal{Q}},\nu))\cdot (132)=0\},$
    then the Lie operad $\mathcal{O}_{Lie}=\langle\mathcal{F}(\mathcal{Q}),\mathcal{R}\rangle$, viewed as free dg operad with differential zero is
    represented as operad generated by $\mathcal{Q}$ subject to the relations in
$\mathcal{R}$.
\end{Ex}

\section{Cosimplicial structure and induced algebraic structures on a free
$\Sigma$-differential graded operad}
 \subsection{Cosimplicial structure on a free $\Sigma$-differential graded operad}
We begin this subsection by adapting some useful results developed in \cite{B-T} to the framework of free $\sum$-dg operads.\\
Let $(\mathcal{O}, d^{\mathcal{O}})$ be a pointed multiplicative
free dg operad with multiplication $\mu\in\mathcal{O}(2)$ and base point
$1_0\in\mathcal{O}(0)$  satisfying
$\mu\circ_11_0=1_{\mathcal{O}}=\mu\circ_21_0$. Assume that
$\{B^{r}_{n}(O); n\geq 0, r\geq 0\}$ is a collection of generators
of $(\mathcal{O}, d^{\mathcal{O}}).$
In the particular case of
operads over the category $Vect_{\mathbb{K}},$ it is well known that
every $\mathbb{K}$-vector space has a basis, finite or infinite. Hence,
such operads are freely generated by the collection of
generators $\{B^{r}_{n}(\mathcal{O}): \mbox{a basis for} ~ ~
(O(n))_{r}; n \geq 0, r\geq 0\}$. Consequently,  the definitions of
the partial composition (respectively, total composition) of operads,
as well as the cosimplical structure defined in \cite{B-T}, can be
restricted to generators as shown below and then extended by linearity:

\begin{Pro} Let $(\mathcal{O}, d^{\mathcal{O}})$ be a pointed multiplicative free dg operad (symmetric or not)
    with multiplication $\mu\in (\mathcal{O}(2))_{0}$ and base point $1_0\in (\mathcal{O}(0))_{0}$
    subject to the following relation: $\mu\circ_11_0=1_{\mathcal{O}}=\mu\circ_21_0$. The
    cofaces and codegeneracies defined below endow $(\mathcal{O}, d^{\mathcal{O}})$ with a differential cosimplical structure.
    \begin{enumerate}
        \item[(a)] The codegeneracies.\\
        For $m,n\ge 0$, $1\le i\le n,$ $1\le j\le m+n-1$, $r, s\geq 0,$ denote by $B_{(\mathcal{O}(m))_{s}}$ (respectively $B_{(\mathcal{O}(n))_{r}}$)
        a given set of generators of $(\mathcal{O}(m))_{s}$ (respectively $(\mathcal{O}(n))_{r}$).
         Consider $x^m\in B_{(\mathcal{O}(m))_{s}}$,  $x^n\in B_{(\mathcal{O}(n))_{r}}$ and define:
          $$(\mathcal{O}(n))_{r}\stackrel{s_i^{n-1}}\longrightarrow(\mathcal{O}(n-1))_{r}~such~ that$$ $$s_i^{n-1}(x^n)=x^n\circ_i1_0,~ s_1^0(1_\mathcal{O})=1_0~ and$$
        $$s_j^{m+n-2}(x^n\circ_i x^m)=\left \{ \begin{array}{lll}
            s_j^{n-1}(x^n)\circ_{i-1}x^m, ~if~ 1\le j< i,\\
            x^n\circ_is_{j-i+1}^{m_1}(x^m), ~if~ i\le j\le i+m,\\
            s_{j-i}^{n-1}(x^n)\circ_{i-1}x^m, ~if~ i+m<j.
        \end{array}
        \right.$$
        \item[(b)] The coface.\\
        For $m,n\ge 0$, $0\le i\le n,$ $0\le j\le m+n$, $r, s\geq 0,$ denote by $B_{(\mathcal{O}(m))_{s}}$ (respectively $B_{(\mathcal{O}(n))_{r}}$), a
         given collection of generators of $(\mathcal{O}(m))_{s}$ (respectively $(\mathcal{O}(n))_{r}$).
         Consider $x^{m}\in B_{(\mathcal{O}(m))_{s}}$,  $x^n\in B_{(\mathcal{O}(n))_{r}}$ and define:
          $$ \begin{array}{ccc}
          D_i^{n+1}: (\mathcal{O}(n))_{r} & \longrightarrow &
          (\mathcal{O}(n+1))_{r}\\
        x^n &\longmapsto & D_i^{n+1}x^n=\left \{ \begin{array}{ll}
             x^n\circ_i \mu, ~if~ 1\le i\le n\\
        \mu\circ_2 x^n, ~if~ i=0\\
        \mu\circ_1 x^n, ~if~ i=n+1
    \end{array}
\right.
\end{array}$$
        with $\mathbb{K}\cong(\mathcal{O}(0))_{0}\stackrel{D^0=\eta}\longrightarrow(\mathcal{O}(1))_{0},$ and  \\
        $$D_j^{m+n}((x^n\circ_ix^{m})=\left \{ \begin{array}{lll}
            D_j^{n+1}(x^n)\circ_{i+1}x^{m}, ~if~ j=0,\\
            D_j^{n+1}(x^n)\circ_ix^{m}, ~if~ j=n+1,\\
            D_{j-m+1}^{n+1}(x^n)\circ_{i}x^{m}, ~if~ m<j\le n+m-1.
        \end{array}
        \right.$$
    \end{enumerate}
\end{Pro}

 \begin{Rem}
 \begin{enumerate}
        \item[(a)] Observe that for any pointed multiplicative free dg operad generated by a collection of
objects subject to some relations, the statements of the above
proposition will be similar with the particularity that generators
will be replaced by their equivalent classes.
 \item[(b)]Furthermore for a given free pointed multiplicative
 $\Sigma$-dg operad $(\mathcal{O},
d^{\mathcal{O}})$ together with its cosimplicial structure, we
recall below some results established in \cite{B-T}:
\begin{enumerate}
        \item[(i)] The map
        $\begin{array}{ccc}
        (\mathcal{O}(n))_{r}& \stackrel{\partial^{r}_{n}}\longrightarrow & (\mathcal{O}(n-1))_{r}\\
       x & \longmapsto &
         \partial^{r}_{n}=\sum\limits_{i=1}^n(-1)^{i} s_{i}^{n-1}
    \end{array}$
        is a boundary operator.
        \item[(ii)] The map
        $\begin{array}{ccc}
        (\mathcal{O}(n))_{r}& \stackrel{\delta^{r}_{n}}\longrightarrow & (\mathcal{O}(n+1))_{r}\\
       x & \longmapsto &
        \delta^{r}_n=\sum\limits_{i=0}^{n+1}(-1)^i D_i^{n+1},
    \end{array}$
    is a coboundary operator.
        \item[(iii)] The triple $(\mathcal{O},\odot, D=\delta +\partial)$ is a double complex algebra, called a bicomplex
        algebra, with the product defined by: $$\begin{array}{ccc}
        \bigoplus\limits_{u+v=r}^{}(\mathcal{O}(p))_{u}\otimes(\mathcal{O}(q))_{v}& \stackrel{\odot}\longrightarrow & (\mathcal{O}(p+q))_{r}\\
       x\otimes y & \longmapsto & x\odot y= \mu(x,y)=(\mu\circ_2 y)\circ_1 x.
    \end{array}$$
 \end{enumerate}
 \end{enumerate}
\end{Rem}
\subsection{Induced algebraic structures on a free
$\Sigma$-differential graded operad.} Hereafter, we endow a free
cosimplicial $\Sigma$-dg operad $(\mathcal{O}, d^{\mathcal{O}})$
with several differential graded algebraic structures, including a
twisted differential graded algebraic structure which, in the
particular non-differential graded case of the associative operad
$\mathcal{A}_{ss},$ coincides with the Malvenuto-Reutenauer
twisted algebra structure(\cite{M-R}).\\
\begin{Pro}
    Let $(\mathcal{O}, d^{\mathcal{O}})$ be a free pointed multiplicative
$\Sigma$-dg operad together with its cosimplicial structure, such
that its multiplication $\mu$ and its base point $1_0$ are subject
to the relation: $\mu\circ_11_0=1_{\mathcal{O}}=\mu\circ_21_0$.
     The following statements hold:
    \begin{enumerate}
\item[(a)] The pair $((\mathcal{O}, d^{\mathcal{O}}), \odot)$ is a chain complex algebra.
\item[(b)] The triple $((\mathcal{O}, d^{\mathcal{O}}), \partial, \odot)$ is a bicomplex algebra.
\item[(c)] The triple $((\mathcal{O}, d^{\mathcal{O}}), \delta, \odot)$ is a bicomplex algebra.
    \end{enumerate}
\end{Pro}
\begin{proof}
\begin{enumerate}
\item[(a)] The fact that the pair $((\mathcal{O}, d^{\mathcal{O}}), \odot)$ is a chain complex algebra follows from Definition 2.3,
 Definition 2.6(1) and  Remark 3.2($b\mbox{-}(iii)$).
\item[(b)] To show that the triple $((\mathcal{O}, d^{\mathcal{O}}), \partial, \odot)$ is a bicomplex algebra, observe that for $n\geq 0, r\geq 0,$
 for $x\in (\mathcal{O}(n))_{r}$, $$d_{r}^{\mathcal{O}(n-1)}(\partial^{r}_{n}(x))= \sum\limits_{i=1}^{n} (-1)^{i} d_{r}^{\mathcal{O}(n-1)}(x \circ_{i} 1_{0})=
 \sum\limits_{i=1}^{n} (-1)^{i} d_{r}^{\mathcal{O}(n)}(x) \circ_{i} 1_{0}=
 \partial^{r-1}_{n}(d_{r}^{\mathcal{O}(n)}(x)).$$ The result follows from
 the previous assertion and Remark 2.16($b\mbox{-}(iii)$).
\item[(c)] For the last assertion, observe that for $n\geq 0, r\geq 0,$
 for $x\in (\mathcal{O}(n))_{r}$,
 $$ \begin{array}{ccc}
        d_{r}^{\mathcal{O}(n+1)}(\delta^{r}_{n}(x)) &=\left \{ \begin{array}{ll}
              \mu\circ_2 d_{r}^{\mathcal{O}(n)}(x^{n}), ~if~ i=0\\d_{r}^{\mathcal{O}(n)}( x^{n}) \circ_i \mu, ~if~ 1\le i\le n\\
        \mu\circ_1 d_{r}^{\mathcal{O}(n)}(x^{n}), ~if~ i=n+1
    \end{array}
\right. =& ~\delta^{r-1}_{n}(d_{r}^{\mathcal{O}(n)}(x^{n})).
\end{array}$$
The result follows again by assertion (a) and Remark 2.16($b\mbox{-}(iii)$).
\end{enumerate}This completes the proof. \end{proof}

 \indent \textbf{Twisted algebraic structure on a free cosimplicial $\Sigma$-dg operad}\\

Let $(\mathcal{O}, d^{\mathcal{O}})$ be a free cosimplicial
$\Sigma$-dg operad. Define a non-commutative product on
$(\mathcal{O}, d^{\mathcal{O}})$ called  shuffle product as follows:
$$\begin{array}{ccc}
    (\mathcal{O}(p) \otimes \mathcal{O}(q))_{r} = \bigoplus\limits_{u + v = r}(\mathcal{O}(p))_{u}\otimes (\mathcal{O}(q))_{v} &
    \stackrel{\nabla}\longrightarrow &(\mathcal{O}(p+q))_{r} \\
    x\otimes y & \longmapsto & x\nabla y= \sum\limits_{\sigma\in \Sigma_{(p,q)}}sgn(\sigma)(x\odot y)\cdot\sigma^{-1}.
 \end{array}$$
 More explicitly
$$\begin{array}{ccc}
    (x\nabla y)(x_1, x_2, \cdots, x_{p+q})&= &\sum\limits_{\sigma\in \Sigma_{(p,q)}} sgn(\sigma)(\mu(x, y)\cdot\sigma^{-1})(x_1, x_2, \cdots, x_{p+q})\\
     &= & \sum\limits_{\sigma\in \Sigma_{(p,q)}}(-1)^{\xi + \epsilon(\sigma)}\mu(x(x_{\sigma(1)},\cdots,x_{\sigma(p)}), y(x_{\sigma(p+1)},\cdots,x_{\sigma(p+q)}))
\end{array}$$
where $x_{i}\in (\mathcal{O}(q))_{r_{i}},$ $\xi=
deg(y)(\sum\limits_{j=1}^{p} deg(x_{j}))$, $sgn(\sigma)=
(-1)^{\epsilon(\sigma)}$ denotes the signature of $\sigma \in
\Sigma_{(p,q)}$, with $\Sigma_{(p,q)}$ being the set of $(p,q)$-shuffles
 (see \cite{E-M}), and the equivariance property
mentioned in Definition 2.3 is explicitly given by:
$$\gamma^{\mathcal{O}}(x\cdot\sigma; x_1, x_2, \cdots,
 x_n):= \gamma^{\mathcal{O}}(x; x_{\sigma^{-1}(1)}, x_{\sigma^{-1}(2)},\cdots, x_{\sigma^{-1}(n)})
 .$$ This product leads to the following statement:
\begin{Pro} A free cosimplicial $\Sigma$-dg operad, $(\mathcal{O}, d^{\mathcal{O}})$,
together with the shuffle product $\nabla$ is a differential graded
algebra. It is called twisted differential graded algebra.
\end{Pro}
\begin{proof}
   Observe that $\nabla$ is associative with unit $1_{0}$ since $\mu\circ_{1}\mu - \mu\circ_{2}\mu=
   0,$ the shuffle is associative
    and $\mu\circ_{1}1_{0}= 1_{\mathcal{O}} =\mu\circ_{2}1_{0}$. Moreover,
   the $\Sigma$-equivariant  linear homomorphism $d^{\mathcal{O}}$ is a derivation with respect to the product $\nabla$.
   Hence $(\mathcal{O}, d^{\mathcal{O}}, \nabla)$ is a differential graded algebra
\end{proof}\\

\section{Differential graded Hopf algebra structure on cosimplicial $\Sigma$-free operads.} In this section, we focus on
the particular case of  free cosimplicial $\Sigma$-dg operads
concentrated in degree zero ($((\mathcal{O}(n))_{r}= 0)_{n\geq 0}^{r>
0}$ and $d=0$).
\subsection{Differential graded coalgebra structure on free cosimplicial $\Sigma$-operads}
This subsection is devoted to the construction of a coaugmented
differential graded coalgebra structure on a free cosimplicial
$\Sigma$-operads $\mathcal{O}$ together with its representation by
generators subject to some relations.\\
For a given integer $k\ge 1$, let us consider the following linear
maps:
\begin{center}
    $\widetilde{s}^{k-j}:
    \mathcal{O}(k)\longrightarrow \mathcal{O}(j)$
\end{center}
and
\begin{center}
    $\widetilde{s}_1^{j}: \mathcal{O}(k)\longrightarrow
    \mathcal{O}(k-j)$
\end{center}
defined on the representatives of the generators as follows:\\ for
any representative $\hat{y}$ of a given generator $y,$ an operation
with $k$ inputs,
\begin{center}

    $$\widetilde{s}^{k-j}\hat{y} = \left\{\begin{array}{ll}s_{j+1}^{j}\circ
        s_{j+2}^{j+1}\circ \cdots \circ s_{k-1}^{k-2}\circ s_{k}^{k-1}\hat{y}, \quad \mbox{if} \quad 0 \leq j<k,\\
        \hat{y},  \quad if \quad j= k. \end{array}
        \right.$$\\
    $$\widetilde{s}_1^{j}\hat{y} =
    \left\{\begin{array}{ll}s_{1}^{k-j}\circ
        s_{1}^{k-j+1}\circ \cdots \circ s_{1}^{k-2}\circ s_{1}^{k-1}\hat{y}, \quad \mbox{if} \quad 0< j\leq k,\\
        \hat{y},  \quad if \quad j= 0.
        \end{array} \right.$$
\end{center}
In addition, these maps satisfy the following relations:\\
For $\hat{x}\in \mathcal{O}(n), ~\hat{y}\in \mathcal{O}(m), ~1\leq i
\leq n, ~1\leq j \leq n + m -1,$
\begin{center}

$$\widetilde{s}^{n+m-j-1}(\hat{x}\circ_{i}\hat{y}) =
\left\{\begin{array}{lll}\widetilde{s}^{n-j}(\hat{x}), \quad \quad \mbox{if}\quad \quad 1 \leq j < i \leq n,\\
       (\widetilde{s}^{n-i}(\hat{x}))\circ_{i} (\widetilde{s}^{m-1}(\hat{y})) ,  \quad \mbox{if} \quad 1 \leq j = i \leq n,\\
(\widetilde{s}^{n-(j+1-m)}(\hat{x}))\circ_{i} \hat{y} , \quad \quad
\mbox{if} \quad 1 \leq  i < j \leq n+m-1,
\end{array}
 \right.$$
\end{center}

$$\mbox{and} \quad \widetilde{s}_{1}^{j}(\hat{x} \circ_{i} \hat{y})= (\widetilde{s}_{1}^{i-1}(\hat{x})) \circ_{1} (\widetilde{s}_{1}^{j-i+1}(\hat{y})).$$
Observe that since the above defined maps are linear by extension, then the following map is also linear by extension.\\
For all $k\geq 0,$
$$\begin{array}{ccc}
        \Delta:\mathcal{O}(k)& \longrightarrow &
        (\mathcal{O}\otimes\mathcal{O})(k)=
        \bigoplus_{j=0}^{n}\mathcal{O}(j)\otimes\mathcal{O}(k-j)\\ \hat{y} &
        \longmapsto &
        \sum_{j=0}^{k}{\widetilde{s}^{k-j}\hat{y}\otimes\widetilde{s}_1^j \hat{y}}.
    \end{array}$$

It is well known that the triple $((\mathcal{O}, \delta),
\odot,\eta)$ is a
graded  differential unitary algebra, with unit $\eta$ and product $\odot$ (see
\cite{B-T}). So we study, through the following
statements, some relations between the linear map $\Delta$ and the
preceding algebra structure.

\begin{Pro}
 Let $\mathcal{O}$ be a free connected multiplicative operad with multiplication $\mu$ and base point $1_{0}$ such that $\mu \circ_{1} 1_{0} =
1_{\mathcal{O}} = \mu \circ_{2} 1_{0}.$  The above constructed
linear map $\Delta$  is  a homomorphism
 of algebras.
\end{Pro}

\begin{proof}
We just have  to check commutativity of the following diagram:
    $$\xymatrix{
        \mathcal{O}\otimes \mathcal{O} \ar [d]^{\Delta\otimes \Delta} \ar[r]^{\odot}&   \mathcal{O} \ar [r]^{\Delta}& \mathcal{O}\otimes \mathcal{O} \\
        \mathcal{O} \otimes \mathcal{O}\otimes \mathcal{O}\otimes\mathcal{O}  \ar@{-}[rr]^{}  &  ^{\quad\quad \quad\quad\quad\quad Id\otimes T\otimes Id} & \quad \mathcal{O} \otimes \mathcal{O}\otimes \mathcal{O}\otimes\mathcal{O}  \ar[u]_{\odot\otimes\odot}
    } $$
that is, $[(\odot\otimes\odot)\circ(Id\otimes $T$\otimes
Id)\circ(\Delta\otimes\Delta)]=\Delta\circ\odot$ where $T$ is the
interchange homomorphism.\\ To achieve this purpose, consider
$x\in\mathcal{O}(n)$, $y\in\mathcal{O}(u)$ and denote by $\hat{x},$
$\hat{y},$ their respective representatives. Set $k= n + u.$ We
have
\begin{align*}
    (\Delta\circ\odot)(\hat{x}\otimes \hat{y})&=\Delta[\mu(\hat{x},\hat{y})]\\
 &=\sum\limits_{j=0}^k\widetilde{s}^{k-j}(\mu(\hat{x},\hat{y}))\otimes \widetilde{s}_1^{j}(\mu(\hat{x},\hat{y})).
\end{align*}
Observe that each term of the above sum can be broken down as follows:\\
if $j\in\{0, 1,\cdots, n-1\}$, then
$$\widetilde{s}^{k-j}(\mu(\hat{x},\hat{y}))=\mu(\widetilde{s}^{n-i}(\hat{x}),\widetilde{s}^{u-t}(\hat{y}))$$ and
$$\widetilde{s}_1^{j}(\mu(\hat{x},\hat{y}))=\mu(\widetilde{s}_1^{i}(\hat{x}),\widetilde{s}_1^{t}(\hat{y}))$$
where $i= j$ and $t= 0.$\\
If $j\in\{n, n+1,\cdots, n + u= k\}$, then
$$\widetilde{s}^{k-j}(\mu(\hat{x},\hat{y}))=\mu(\widetilde{s}^{n-i}(\hat{x}),\widetilde{s}^{u-t}(\hat{y}))$$ and
$$\widetilde{s}_1^{j}(\mu(\hat{x},\hat{y}))=\mu(\widetilde{s}_1^{i}(\hat{x}),\widetilde{s}_1^{t}(\hat{y}))$$
where $i=n$ and $t=j-n.$\\
Thus the above sum becomes:
\begin{align*}
    (\Delta\circ\odot)(x\otimes \hat{y})&=\Delta[\mu(\hat{x},\hat{y})]\\
    &=\sum\limits_{j=0}^k \widetilde{s}^{k-j}(\mu(\hat{x},\hat{y}))\otimes \widetilde{s}_1^{j}(\mu(\hat{x},\hat{y}))\\
    &=\sum\limits_{j=0}^k \mu(\widetilde{s}^{n-i}(\hat{x}),\widetilde{s}^{u-t}(\hat{y}))\otimes \mu(\widetilde{s}_1^{i}(\hat{x}),\widetilde{s}_1^{t}(\hat{y}))\\
    &=\sum\limits_{i=0}^n \sum\limits_{t=0}^u \mu(\widetilde{s}^{n-i}(\hat{x}),\widetilde{s}^{u-t}(\hat{y}))\otimes \mu(\widetilde{s}_1^{i}(\hat{x}),\widetilde{s}_1^{t}(\hat{y}))\\
    &=(\odot\otimes\odot)(\sum\limits_{i=0}^n\sum\limits_{t=0}^u \widetilde{s}^{n-i}(\hat{x})\otimes\widetilde{s}^{u-t}(\hat{y})\otimes \widetilde{s}_1^{i}(\hat{x})\otimes\widetilde{s}_1^{t}(\hat{y}))\\
    &=[(\odot\otimes\odot)\circ(Id\otimes T\otimes Id)](\sum\limits_{i=0}^n\sum\limits_{t=0}^u \widetilde{s}^{n-i}(\hat{x})\otimes\widetilde{s}_1^{i}(\hat{x})\otimes \widetilde{s}^{u-t}(\hat{y})\otimes\widetilde{s}_1^{t}(\hat{y}))\\
    &=[(\odot\otimes\odot)\circ(Id\otimes T\otimes Id)](\sum\limits_{i=0}^n\widetilde{s}^{n-i}(\hat{x})\otimes\widetilde{s}_1^{i}(\hat{x}))\otimes(\sum\limits_{t=0}^u \widetilde{s}^{u-t}(\hat{y})\otimes\widetilde{s}_1^{t}(\hat{y}))\\
    &=[(\odot\otimes\odot)\circ(Id\otimes T\otimes Id)\circ(\Delta\otimes \Delta)](\hat{x}\otimes \hat{y}).
\end{align*}
This completes the proof.
\end{proof}

\begin{Pro} Let $\mathcal{O}$ be a free connected multiplicative operad with multiplication $\mu$ and
 base point $1_{0}$ such that $\mu \circ_{1} 1_{0} =
1_{\mathcal{O}} = \mu \circ_{2} 1_{0}.$ The triple $((\mathcal{O},
\delta), \Delta, \epsilon)$ is a coaugmented differential graded
coalgebra.
\end{Pro}
\begin{proof}
 \indent $(i)$ Since $\mathcal{O}$ is a free pointed operad, the
injective homomorphism $\mathbb{K}\stackrel{\eta}\longrightarrow
\mathcal{O}$ is the coaugmentation and its retraction,
$\mathcal{O}\stackrel{\epsilon}\longrightarrow\mathbb{K}$ is the
counit.\\

 \indent $(ii)$ Next we verify the coassociative property of the coproduct, $\Delta.$ This is equivalent to verifying that the following diagram
        $$\xymatrix{
            \mathcal{O}\ar[d]_{\Delta} \ar[r]^{\Delta}& \mathcal{O}\otimes\mathcal{O} \ar[d]^{id\otimes\Delta}  &\\
            \mathcal{O}\otimes\mathcal{O} \ar[r]^{\Delta\otimes id}& (\mathcal{O}\otimes\mathcal{O})\otimes\mathcal{O} &\\
        }$$
         commutes, that is,  $(id\otimes\Delta)\circ\Delta=(\Delta\otimes id)\circ\Delta.$ \\ Indeed, for $y\in\mathcal{O}(k)$
          and $\hat{y}$ its representative, we have \\
        \begin{align*}
            [(id\otimes\Delta)\circ \Delta](\hat{y})&=(id\otimes\Delta)(\sum\limits_{i=0}^k\widetilde{s}^{k-i}\hat{y}\otimes\widetilde{s}_1^{i}\hat{y})\\
            &=\sum\limits_{i=0}^k\widetilde{s}^{k-i}\hat{y}\otimes\Delta(\widetilde{s}_1^{i}\hat{y})\\
            &=\sum\limits_{i=0}^k\widetilde{s}^{k-i}\hat{y}\otimes\sum\limits_{j=0}^{k-i}
            \widetilde{s}^{k-i-j}\widetilde{s}_1^{i}\hat{y}\otimes\widetilde{s}_1^{j}
            \widetilde{s}_1^{i}\hat{y}\\
            &=\sum\limits_{i=0}^k\sum\limits_{j=i}^{k}\widetilde{s}^{k-i}\hat{y}\otimes
            \widetilde{s}^{k-j}\widetilde{s}_1^{i}\hat{y}\otimes\widetilde{s}_1^{j-i}
            \widetilde{s}_1^{i}\hat{y}\\
            &=\sum\limits_{i=0}^k\sum\limits_{j=i}^{k}P_{i,j}.
        \end{align*}
        Similarly, we have:\\
        \begin{align*}
            [(\Delta\otimes id)\circ \Delta](\hat{y})&=(\Delta\otimes id)(\sum\limits_{i=0}^k\widetilde{s}^{k-i}\hat{y}\otimes\widetilde{s}_1^{i}\hat{y})\\
            &=\sum\limits_{i=0}^k\Delta(\widetilde{s}^{k-i}\hat{y})\otimes\widetilde{s}_1^{i}\hat{y}\\
            &=\sum\limits_{i=0}^k\sum\limits_{j=0}^{i}\widetilde{s}^{i-j}\widetilde{s}^{k-i}\hat{y}\otimes
            \widetilde{s}_1^{j}\widetilde{s}^{k-i}\hat{y}\otimes\widetilde{s}_1^{i}\hat{y}\\
            &=\sum\limits_{i=0}^k\sum\limits_{j=0}^{i}Q_{i,j}.
        \end{align*}
        Observe that the exact number of terms of each of the above two summations
        is:
        $$\sum\limits_{i=0}^{k}k-i+1=\sum\limits_{i=1}^{k+1}i=\frac{(k+1)(k+2)}{2}.$$

        Moreover, one can verify  the following equalities by direct computation:

        \begin{enumerate}
            \item[$a)$] $P_{i,i}=Q_{i,i},\quad 0\le i\le k$,
            \item[$b)$] $P_{0,j}=Q_{n,j}, \quad 1\le j\le k-1$,
            \item[$c)$] $P_{i,j}=Q_{j,i}, \quad i<j$.
        \end{enumerate}
       Hence, we obtain the identity  $(id\otimes\Delta)\circ\Delta=(\Delta\otimes id)\circ\Delta,$ and
        therefore the coproduct $\Delta$ is coassociative.\\
 \indent $(iii)$  Finally, we show that the coboundary operator $\{\mathcal{O}(n)\stackrel{\delta^{n}}\longrightarrow \mathcal{O}(n+1): n\geq 0\}$ is
        a coderivation with respect to the coproduct $\Delta.$\\

Before beginning this part of the proof, we recall below some
properties of the cosimplicial structure associated with any pointed multiplicative
operad $\mathcal{O}$:
$$s_i^nD_j^{n+1}=\left\{\begin{array}{lll}
    D_{j-1}^{n+1}s_i^n \quad if \quad i<j,\\
    id \quad if \quad i\in\{j,j+1\},\\
    D_{j}^{n+1}s_{i-1}^n \quad if \quad i>j+1.
   \end{array}
   \right.$$
Thus, for $x\in \mathcal{O}(n)$ and $\hat{x}$ a representative of $x,$ we
have the following statements:
\begin{enumerate}
    \item[(a)] if $i< k$, then
    $$\tilde{s}^{n-k+1}D_i^{n+1}(\hat{x})=s_{k+1}s_{k+2}\cdots s_{n+1}D_i^{n+1}(\hat{x})=D_i^{k+1}s_{k+1}s_{k+2}\cdots s_{n}(\hat{x})=
    D_i^{k}\tilde{s}^{n-k}(\hat{x})$$
    and $$\tilde{s}_1^k D_i^{n+1}(\hat{x})=\underbrace{s_1s_1\cdots s_1}_{k\mbox{-}times}D_i^{n+1}(\hat{x})=\tilde{s}_1^k(\hat{x}),$$
    \item[(b)] if $i\ge k,$ then
    $$\tilde{s}^{n-k+1}D_i^{n+1}(\hat{x})=\tilde{s}^{n-k}(\hat{x})$$
    and $$\tilde{s}_1^k D_i^{n+1}(\hat{x})=\underbrace{s_1s_1\cdots s_1}_{k\mbox{-}times}D_i^{n+1}(\hat{x})=D_{i-k}^{n-k+1}\tilde{s}_1^k(\hat{x}).$$
\end{enumerate}
Moreover, we have on one hand:
\begin{align*}
    [(\delta\otimes id+id\otimes \delta)\circ\Delta](\hat{x})&=(\delta\otimes id+id\otimes \delta)(\sum\limits_{j=0}^n\tilde{s}^{n-j}(\hat{x})\otimes\tilde{s}_1^{j}(\hat{x})\\
    &=\sum\limits_{j=0}^n[\delta(\tilde{s}^{n-j}(\hat{x}))\otimes\tilde{s}_1^{j}(\hat{x})+(-1)^j \tilde{s}^{n-j}(x)\otimes \delta(\tilde{s}_1^{j}(\hat{x}))]\\
    &=\sum\limits_{j=0}^n[(\sum\limits_{i=0}^{j+1}D_i^{j+1}(\tilde{s}^{n-j}(\hat{x})))\otimes\tilde{s}_1^{j}(\hat{x})+(-1)^j \tilde{s}^{n-j}(\hat{x})\otimes (\sum\limits_{i=0}^{n-j+1}D_i^{n-j+1}(\tilde{s}_1^{j}(\hat{x})))]\\
    &=\sum\limits_{j=0}^n[(\sum\limits_{i=0}^{j+1}D_i^{j+1}(\tilde{s}^{n-j}(\hat{x})))\otimes\tilde{s}_1^{j}(\hat{x})]+\sum\limits_{j=0}^n[(-1)^j \tilde{s}^{n-j}(\hat{x})\\
    &~~~~\otimes (\sum\limits_{i=0}^{n-j+1}D_i^{n-j+1}(\tilde{s}_1^{j}(\hat{x})))].
\end{align*}
On the other hand, we also have:
\begin{align*}
    (\Delta\circ \delta)(\hat{x})&=\sum\limits_{k=0}^{n+1}\tilde{s}^{n-k+1}(\delta(\hat{x}))\otimes\tilde{s}_1^{k}(\delta(\hat{x}))\\
    &=\sum\limits_{k=0}^{n+1}\tilde{s}^{n-k+1}(\sum\limits_{i=0}^{n+1}(-1)^iD_i^{n+1}(\hat{x}))\otimes\tilde{s}_1^{k}(\sum\limits_{i=0}^{n+1}(-1)^iD_i^{n+1}(\hat{x}))\\
    &=\sum\limits_{k=0}^{n+1}[(\sum\limits_{i=0}^{n+1}(-1)^i\tilde{s}^{n-k+1}D_i^{n+1}(\hat{x}))\otimes (\sum\limits_{i=0}^{n+1}(-1)^i\tilde{s}_1^{k}D_i^{n+1}(\hat{x}))]\\
    &=\sum\limits_{k=0}^n[(\sum\limits_{i=0}^{k+1}D_i^{k+1}(\tilde{s}^{n-k}(\hat{x})))\otimes\tilde{s}_1^{k}(\hat{x})]+\sum\limits_{k=0}^n[(-1)^j \tilde{s}^{n-k}(\hat{x})\\&~~~~\otimes (\sum\limits_{i=0}^{n-k+1}D_i^{n-k+1}(\tilde{s}_1^{k}(\hat{x})))].\\
\end{align*}
So, putting the above established identities together with the
statements $(a)$ and $(b)$ yields: $\Delta\circ \delta
=[(\delta\otimes id + id\otimes \delta)\circ\Delta]$
\end{proof}\\
\subsection{ Differential graded Hopf algebra structure on free connected
multiplicative operads and some induced structures.} In this subsection, we endow a free connected multiplicative operad with a differential graded Hopf algebra structure and investigate some of the resulting structures.

\begin{Th}
    Let $\mathcal{O}$ be a connected multiplicative operad with multiplication $\mu$ and
 base point $1_{0}$ such that $\mu \circ_{1} 1_{0} =
1_{\mathcal{O}} = \mu \circ_{2} 1_{0}.$  The quintuplet
$((\mathcal{O}, \delta), \odot,\eta, \Delta,\epsilon)$  is a
differential graded Hopf algebra.
\end{Th}
\begin{proof}
It is a routine exercise to establish that the counit $\epsilon$ is
a homomorphism of algebras. Now, since our operad is connected
and multiplicative, we define the antipodal homomorphism,
$\mathcal{O}\stackrel{\mathcal{A}}\longrightarrow\mathcal{O},$
recursively
on the number of inputs of the generators as follows:\\
Let $\hat{x}$ be a generic element and denote by $n_{\hat{x}}$ its
number of inputs,
\begin{center}
$$ \mathcal{A}(\hat{x}) =
\left\{\begin{array}{ll} -\hat{x} \quad \quad \mbox{if}\quad \quad   n_{\hat{x}}\in \{0; 1\},\\
       -\hat{x} - \sum\limits_{j=0}^{n_{\hat{x}}} (\mu\circ_{1}\mathcal{A}(\widetilde{s}^{n-j}(\hat{x})))\circ_{2}\widetilde{s}_{1}^{j}(\hat{x}),  \quad
       \mbox{otherwise}.
\end{array}
 \right.$$
\end{center}
 We claim that with the homomorphism $\mathcal{A},$ all the antipodal identities are
 satisfied. In other words, $\mathcal{A},$ is an antipodal
 homomorphism.
We complete the proof by using Proposition 3.1, Proposition 3.2 and Proposition 2.17.
\end{proof}

Let us recall below the definition of the convolution algebra.
\begin{De} Let $(\mathcal{H}, \mu_{\mathcal{H}}, \eta_{\mathcal{H}}, \Delta_{\mathcal{H}}, \epsilon_{\mathcal{H}}, \mathcal{A}_{\mathcal{H}})$
 be a Hopf algebra with unit $\eta_{\mathcal{H}},$ counit $\epsilon_{\mathcal{H}}$ and antipodal map $\mathcal{A}_{\mathcal{H}}.$
The convolution product is the linear map $"\ast"$ defined by:
$$\begin{array}{ccc}
        \mbox{End}(\mathcal{H})\otimes \mbox{End}(\mathcal{H})& \stackrel{\ast}\longrightarrow &
        \mbox{End}(\mathcal{H}) \\ u\otimes v &
        \longmapsto & u\ast v=
        \mu_{\mathcal{H}} \circ (u\otimes v )\circ\Delta_{\mathcal{H}}.
    \end{array}$$
\end{De}

\indent Let $\mathcal{O}$ be a free connected multiplicative operad
endowed with its differential graded Hopf algebra structure.
$$\text{Set}\quad \mbox{End}((\mathcal{O},
\delta))=\{\mbox{End}((\mathcal{O}(n), \delta))=
Hom_{\mathbb{K}}((\mathcal{O}(n), \delta), (\mathcal{O}(n),
\delta)), n\geq 0\},$$ the ring of cochain homomorphisms defined on
$\mathcal{O}.$ We have the following result:
\begin{Pro}\label{4.5}
 The triple $(\mbox{End}((\mathcal{O}, \delta)),\ast, \eta \circ \epsilon)$ is an
associative unitary algebra, called convolution algebra, with unit
$\eta \circ \epsilon.$
\end{Pro}

\begin{proof}
It suffices to show
that the convolution product $"\ast"$ is an inner operator,  since showing the remaining properties of the algebra structure is a
routine exercise.  We show that
for all $u, v \in\mbox{End}((\mathcal{O}, \delta)),$
$$\delta\circ(u\ast v)= (u\ast v)\circ \delta.$$ Indeed, \begin{align*}
    \delta\circ(u\ast v)&=\delta\circ(\mu \circ (u\otimes v )\circ\Delta)\\
    &=\mu \circ (\delta\otimes id + id \otimes \delta) \circ (u\otimes v )\circ\Delta \\
    &=\mu \circ ((\delta\circ u)\otimes v + u \otimes (\delta\circ v))\circ\Delta \\
    &=\mu \circ ((u\circ \delta)\otimes v + u \otimes (v\circ \delta))\circ\Delta \\
    &= (\mu \circ (u\otimes v)\circ \Delta)\circ \delta\\
    &= (u\ast v)\circ \delta.
\end{align*}
This completes the proof.%Hence, $(\mbox{End}((\mathcal{O}, \delta)),\ast, \eta \circ \epsilon)$ is a convolution algebra.
\end{proof}

Note that the convolution product of Proposition \ref{4.5} leads to some well
known operations, namely  Adams operations (\cite{P}) or
characteristic operations (\cite{M-M}), which commute in this case with the coboundary operator. \\
\begin{De}
Let $\mathcal{O}$ be a free connected multiplicative operad endowed with
its differential graded Hopf algebra structure, and
$(\mbox{End}((\mathcal{O}, \delta)),\ast, \eta \circ \epsilon)$ its
associated convolution algebra. The Adams cochain operations are the operations defined recursively
as follows:

$$\left\{\begin{array}{lll} \varphi^{0}& = & \eta\circ \epsilon,\\
& &\\
 \varphi^{k}& = & \underbrace{I \ast I \ast \cdots \ast
I}\limits_{\mbox{k fois}},  \quad
       \mbox{if}  \quad k\geq 1,\quad \mbox{I the linear identity
       in}~\mbox{End}(\mathcal{O}).
\end{array}
 \right.$$
\end{De}
\begin{Co}
Let $\mathcal{O}$ be a free connected multiplicative operad and
$(\mbox{End}((\mathcal{O}, \delta)),\ast, \eta \circ \epsilon),$ its
associated convolution algebra. The following assertions hold:
\begin{enumerate}
\item[$(i)$]  The Adams cochain operations satisfy the following
identity :\\ $$\mbox{for}\quad r, s \in \mathbb{N}, \varphi^{r} \ast
\varphi^{s}= \varphi^{r+s}.$$

\item[$(ii)$]  The set of Adams cochain operations is  a  monoid with respect to the
convolution product.
\item[$(iii)$] The convolution product induces a Lie algebra structure
on   $\mbox{End}((\mathcal{O}, \delta))$ with bracket:
$$\begin{array}{ccc}
        \mbox{End}((\mathcal{O}, \delta))\otimes \mbox{End}((\mathcal{O}, \delta))& \stackrel{[-, -]}\longrightarrow &
        \mbox{End}((\mathcal{O}, \delta))\\ f\otimes g &
        \longmapsto & [f, g]=
        f\ast g - g\ast f.
    \end{array}$$
\end{enumerate}
\end{Co}

\section{  Application on $\mathcal{A}_{ss}=\bigoplus\limits_{n\ge 0}\mathbb{K}[\sum_n]$}
In this section, we apply our construction to the specific operad
$\mathcal{A}_{ss}.$\\ Recall that $\mathcal{A}_{ss}$
is a connected multiplicative operad with multiplication
$\mu=(12)\in \sum_2$ and  base point $1_\mathbb{K}.$ Furthermore, this symmetric operad is
freely generated by $ \mathbb{K}[\sum_2],$  and  admits a presentation by generators and relations. More
precisely $\mathcal{A}_{ss}= \langle\mathbb{K}[\sum_2],
\mathcal{R}\rangle,$ where the set of relations is given by
$$\mathcal{R}=\{\mu\circ_{1} \mu - \mu\circ_{2} \mu,
\mu\circ_{1}1_{\mathbb{K}} - 1_{\mathcal{A}_{ss}}, \mu\circ_{2}
1_{\mathbb{K}} - 1_{\mathcal{A}_{ss}}\}.$$
\subsection{Cosimplicial structure on $\mathcal{A}_{ss}=\bigoplus\limits_{n\ge 0}\mathbb{K}[S_n]$}

The cosimplicial structure on $\mathcal{A}_{ss}$ is defined here
through the partial composition. Recall below some
tools needed for this definition.
\begin{De}{(see \cite{M-R}, \cite{C-P})}
    The standardization map, is the map which, assigns to any word of length p, $X=a_1\cdots a_p$ in the letters $1,\cdots,p$,
    the unique word $st(X)=b_1\cdots b_p$ in the letters $1,\cdots,p$,
   without repetition of the letters, such that the relative order of the letters
   is preserved: $b_i<b_j$ if $a_i<a_j.$ Furthermore, if $a_i= a_j$
    and $i<j$, then $b_i<b_j$.
\end{De}
\begin{Ex} $st(3645)=(1423),$  $st(2122)=(2134).$
\end{Ex}
\indent Recall that $\mathcal{\sum}= \{\sum_{n}\}_{n\geq 1}$ is a sequence
of sets whose $n^{th}$ term is the set of permutations of order $n$.
For every $1\leq i\leq n,$ we define a partial composition on $\mathcal{\sum}=
\{\sum_{n}\}_{n\geq 1}$ as follows: $$\begin{array}{ccc}
\sum_{n} \times \sum_{l} &
    \stackrel{\circ_i}\longrightarrow &
    \sum_{n+l-1}\\
    (\tau, \sigma) & \mapsto & \tau \circ_i \sigma
\end{array}.$$ We recall below the method proposed in (see \cite{B-T}) to compute $\tau \circ_i \sigma$:
\begin{enumerate}
\item[(i)] We start with the computation of $\tau(i)$ by considering the block of length $l$ $$A=(\tau(i),\quad \tau(i)+1, \quad ...\quad, \tau(i)+l-1)$$
which is permuted by $\sigma$. The result obtained is denoted
$\sigma(A)$.
\item[(ii)] Next we consider a subdivision $B,$ of  $[n+l-1]$ into $n$ blocks in which the block $\sigma(A)$ is located at the position $\tau(i)$.
\item[(iii)] Finally $\tau\circ_i\sigma=\tau^{B},$ the permutation of $B$ by
$\tau$; that is, we permute the $n$ blocks of $B$ as components of
the permutation by $\tau.$
\end{enumerate}
\begin{Ex} Let $\tau=(4312)\in \sum_{4}$ and $\sigma=(231)\in
\sum_{3}$, then we are going to compute
consecutively $\tau\circ_1\sigma\in \sum_{6}$ and $\tau\circ_2\sigma\in\sum_6$:\\
$\bullet$ Since $\tau(1)=4$ and $A=(456),$ then $\sigma(A)=(564)$.
So, the subdivision of $[6]$ into 4 blocks is
$B=(1)(2)(3)\sigma(A)=(1)(2)(3)(564).$
 It follows that  $\tau(B)=(564312)=\tau\circ_1\sigma$.\\
$\bullet$ Furthermore, $\tau(2)=3$  and $A=(345)$ implies that
$\sigma(A)=(453)$. Now, considering the subdivision of $[6]$ into 4
blocks $B=(1)(2)\sigma(A)(6)=(1)(2)(453)(6),$ we obtain
$\tau(B)=(645312)=\tau\circ_2\sigma$.
\end{Ex}
 Hereafter, we define the maps that endow the connected  multiplicative free generated operad $\mathcal{A}_{ss}$ with a
 cosimplicial structure.
\begin{enumerate}
    \item Codegeneracies on $\mathcal{A}_{ss}$:\\
     for $1\le i\le n$, we have
   $$\begin{array}{ccc}
    \mathbb{K}[\sum_n]&\stackrel{s^{i,n-1}}\longrightarrow & \mathbb{K}[\sum_{n-1}]\\
    \sigma & \longmapsto & s^{i, n-1}(\sigma)=\sigma\circ_i1_0
    \end{array}$$
    such that $$\sigma\circ_i1_0=st(\sigma (1)\cdots\sigma(i-1)\sigma(i+1)\cdots\sigma(n)).$$
    \item Cofaces on $\mathcal{A}_{ss}$:\\ for $0\le i\le n+1$, we have
    $$\begin{array}{ccc}
        \mathbb{K}[\sum_n] & \stackrel{d^{i,n+1}}\longrightarrow & \mathbb{K}[\sum_{n+1}]\\
    \sigma & \longmapsto & d^{i, n+1}\sigma
   \end{array}$$
    such that $$d^{i, n+1}(\sigma)=\left \{ \begin{array}{ll}\sigma\circ_i\mu,~ if~ 1\le i\le n\\ \mu\circ_2\sigma, ~if~ i=0\\ \mu\circ_1\sigma, ~if~ i=n+1.
\end{array}
 \right.$$
\end{enumerate}

\begin{Pro}(\cite{B-T})\\ The triple $(\mathcal{A}_{ss}, \odot, D=\delta +\partial)$
with
 $$\mathcal{A}_{ss}=\{\mathcal{A}_{ss}(n)=\mathbb{K}[\sum_{n} ]\}_{n\ge 0}, ~\sigma\odot \tau= \sigma \star \tau, ~~\mbox{the concatenation of}~~ \sigma ~~ \mbox{and} ~~\tau;$$
$$\delta_n=\sum\limits_{i=0}^{n+1}(-1)^id^{i,n+1}:\mathbb{K}[ \sum_{n} ]\longrightarrow\mathbb{K}[ \sum_{n+1} ], ~~n\ge 0;$$
$$\partial_n=\sum\limits_{i=1}^{n}(-1)^is^{i,n-1}:\mathbb{K}[ \sum_{n} ]\longrightarrow\mathbb{K}[ \sum_{n-1} ], ~~n\ge 1$$
  is a bicomplex algebra.
\end{Pro}

\subsection{Differential graded Hopf algebra on $\mathcal{A}_{ss}$}

we begin this subsection with the definition of a coproduct on the
specific connected multiplicative free generated operad
$\mathcal{A}_{ss}$ as follows: for a given nonzero integer $n,$
 $$\begin{array}{ccc}
        \Delta:\mathcal{A}_{ss}(n)& \longrightarrow &
        (\mathcal{A}_{ss}\otimes\mathcal{A}_{ss})(n)=
        \bigoplus\limits_{p+q=n}\mathcal{A}_{ss}(p)\otimes\mathcal{A}_{ss}(q)\\ \tau &
        \longmapsto &
       \Delta(\tau)=\sum_{j=0}^{n}{\widetilde{s}^{n-j}\tau\otimes\widetilde{s}^{1,j}\tau}
       \end{array} $$
       where the maps $\widetilde{s}^{n-j}:
       \mathcal{A}_{ss}(n)\longrightarrow\mathcal{A}_{ss}(j)$ and
       $\widetilde{s}^{1,j}: \mathcal{A}_{ss}(n)\longrightarrow
       \mathcal{A}_{ss}(n-j)$ are respective defined as follows:
       $$\widetilde{s}^{n-j} = \left\{\begin{array}{ll}s^{j+1,j}\circ
        s^{j+2,j+1}\circ \cdots \circ s^{n-1,n-2}\circ s^{n,n-1}, \quad if \quad 0 \leq j<n,\\
        Id,  \quad if \quad j= n \end{array} \right.$$ and
       $$\widetilde{s}^{1,j}=
       \left\{\begin{array}{ll}
        s^{1,n-j}\circ
        s^{1,n-j+1}\circ \cdots \circ s^{1,n-2}\circ s^{1,n-1},~ if ~ 0< j\le n,\\
        Id, ~~~~if~~~~~ j= 0. \end{array}
        \right.$$
        In addition, this coproduct leads to the result below which is
         a consequence of Proposition 3.2.
\begin{Pro}
    The triple $((\mathcal{A}_{ss}, \delta)=\bigoplus\limits_{n\ge 0}(\mathcal{A}_{ss}(n), \delta),\Delta,\varepsilon)$
    is a bicomplex coalgebra with counit, $\mathcal{A}_{ss}\stackrel{\varepsilon}\longrightarrow\mathbb{K},$
 the identity homomorphism in degree zero and zero homomorphism in nonzero degree.
\end{Pro}

\begin{Rem} Consider a permutation $\sigma=(\sigma(1)~\cdots~\sigma(n))\in
\sum_n$, and write for $1\le i\le n,$
$$\sigma_i^-=(\sigma(1)~\cdots~\sigma(i))~~\mbox{and}~~\sigma_{i+1}^+=(\sigma(i+1)~\cdots~\sigma(n).$$
The Malvenuto-Reutenauer coassociative coproduct (see \cite{M-R} for
more details) is defined as follows: for all $\sigma\in \mathbb{K}[
\sum_{n} ]$
\begin{center}
    $\Delta_{MR}(\sigma)=\sum\limits_{i=0}^nst(\sigma_i^-)\otimes st(\sigma_{i+1}^+)$.
\end{center}
The associated counit is:
\begin{center}
    $\varepsilon_{MR}:\mathbb{K}[ \sum_{n} ]\longrightarrow\mathbb{K}$
\end{center}
such that $\varepsilon_{MR}^0=id$ and $\varepsilon_{MR}^k=0$, for all
$k\neq 0$.\\ Moreover,  the coproducts $\Delta$ and $\Delta_{MR}$
coincide (see [\cite{C-P}, proposition 9.7.4] or [\cite{M-R}, section
3, page 977], when  the field $\mathbb{K}$ is replaced by the ring
$\mathbb{Z}$).
\end{Rem}
 From the identification of the previous Remark, the Malvenuto-Reutenauer result
 becomes:
\begin{Pro}
The quintuple $((\mathcal{A}_{ss}, \delta), \odot=\star , \Delta=
\Delta_{M-R}, \eta, \varepsilon)$ is a differential graded Hopf
algebra.
\end{Pro}

\end{document}